\documentclass[a4paper,12pt]{article}
\usepackage[top=2.5cm,bottom=2.5cm,left=2.5cm,right=2.5cm]{geometry}
\usepackage{cite,amsmath,amssymb}
    \usepackage[margin=1cm,%
                font=small,%
                format=hang,%
                labelsep=period,%
                labelfont=bf]{caption}
\usepackage[algoruled,linesnumbered]{algorithm2e}
\usepackage{algorithmic}
\newcommand{\Cl}{\rm Cl}

\usepackage{amsfonts,epsf,here,graphicx}
\usepackage{latexsym,multirow,rotating}
\usepackage{pgf,tikz,subfigure}
\usepackage{epstopdf}

\newtheorem{theorem}{\bf Theorem}[section]

\newtheorem{lemma}[theorem]{\bf Lemma}

\newtheorem{definition}[theorem]{\bf Definition}

\newcommand{\qed}{\hfill $\square$ \bigskip}

\begin{document}

\baselineskip=0.30in
\vspace*{40mm}

\begin{center}
{\LARGE \bf On the Clar Number of Benzenoid Graphs}
\bigskip \bigskip

{\large \bf Nino Ba\v si\' c$^{a}$ \qquad Istv\'{a}n Est\'{e}lyi$^b$ \qquad Riste \v Skrekovski$^{c,d}$ \qquad Niko Tratnik$^e$
}
\bigskip\bigskip

\baselineskip=0.20in
$^a$\textit{Faculty of Mathematics, Natural Sciences and Information Technologies, University of Primorska, Slovenia} \\
{\tt nino.basic@famnit.upr.si}
\medskip

$^b$\textit{NTIS, University of West Bohemia, Czech Republic
}\\
{\tt estelyii@gmail.com}
\medskip

$^c$ \textit{Faculty of Mathematics and Physics, University of Ljubljana, Slovenia} \\
{\tt riste.skrekovski@fmf.uni-lj.si}
\medskip

$^d$ \textit{Faculty of information studies, Novo mesto, Slovenia} \\
\medskip

$^e$ \textit{Faculty of Natural Sciences and Mathematics, University of Maribor, Slovenia} \\
{\tt niko.tratnik@um.si}
\medskip

\bigskip\medskip

(Received \today)

\end{center}

\noindent
\begin{center} {\bf Abstract} \end{center}

\vspace{3mm}\noindent
A Clar set of a benzenoid graph $B$ is a maximum set of independent alternating hexagons over all perfect matchings of $B$. The Clar number of $B$, denoted by ${\Cl}(B)$, is the number of hexagons in a Clar set for $B$. In this paper, we first prove some results on the independence number of subcubic trees to study the Clar number of catacondensed benzenoid graphs. As the main result of the paper we prove an upper bound for the Clar number of catacondensed benzenoid graphs and characterize the graphs that attain this bound. More precisely, it is shown that for a catacondensed benzenoid graph $B$ with $n$ hexagons ${\Cl}(B) \leq [(2n+1)/3]$.




\section{Introduction}

The Clar number of a molecular graph $G$ (for example benzenoid graph, fullerene or carbon nanotube) is the maximum number of independent alternating hexagons over all perfect matchings of $G$. This concept originates from Clar's aromatic sextet theory \cite{clar} and has been studied in many papers for benzenoid graphs \cite{ha2,sa-gu} and fullerenes \cite{ahmadi,har,ye-zh,zh-ye}. Also, the connections between the Clar number and the Fries number (i.e. the maximum number of alternating hexagons over all perfect matchings) were investigated in \cite{gra,har1} and some relations to linear programming were considered in \cite{abeledo,ha}. In \cite{kl-zi-gu} an algorithm for computing the Clar number of a catacondensed benzenoid graph was proposed. 

Moreover, matchings and perfect matchings of a molecular graph play an important role in many fields of chemical graph theory. For example, they are essentially used in studying resonance graphs \cite{dos}, saturation number \cite{ahmadi2}, enumeration of matchings \cite{ewi}, Hosoya index \cite{hos2,hos}, Zhang-Zhang polynomial \cite{zhang_pol}, forcing and anti-forcing numbers \cite{shi}, internal Kekul\' e structures \cite{gra1}, etc. Furthermore, there are connections between resonance graphs and Clar sets \cite{sa-kl}.

In the present paper, we prove an upper bound for the Clar number of catacondensed benzenoid graphs and characterize the graphs that attain the bound. We proceed as follows. In the following section we first formally define all the important concepts. Since the problem of studying the Clar number of a catacondensed benzenoid graph can be transformed into the problem of studying the independence number of its dualist tree, in Section \ref{trees} we prove some results on the independence number of subcubic trees. Finally, in Section \ref{bound} we prove the upper bound and characterize all extremal graphs with respect to this bound.
\section{Preliminaries}

In the existing (both mathematical and chemical) literature, there is inconsistency in the terminology pertaining to (what we call
here) ``benzenoid graph". In order to avoid any confusion, we first define our objects.

A \textit{benzenoid graph} is a 2-connected graph
in which all inner faces are hexagons (and all hexagons are faces), such that two
hexagons are either disjoint or have exactly one common edge, and no three hexagons
share a common edge.

Note that in some literature it is assumed that a benzenoid graph can be embedded into the regular hexagonal lattice \cite{gucy-89}. Obviously, our definition is more general and includes graphs that cannot be embedded into the regular hexagonal lattice. For more details on these definitions see \cite{DGKZ-2002}.

Let $B$ be a benzenoid graph. A vertex shared by three hexagons of $B$ is called an \textit{internal} vertex of $B$. A benzenoid graph is said to be \textit{catacondensed} if it does not possess internal vertices. Otherwise it is called \textit{pericondensed}.

A \textit{matching} $M$ in a graph $G$ is a set of edges of $G$ such that no two edges from $M$ share a vertex. If every vertex of $G$ is incident with an edge of $M$, the matching $M$ is called a \textit{perfect matching} (in chemistry perfect matchings are known as \textit{Kekul\'{e} structures}). Let $B$ be a benzenoid graph and $h$ a hexagon of $B$. If $M$ is a matching that contains exactly 3 edges of $h$, then $h$ is an $M$-\textit{alternating} hexagon. In such cases we often draw a circle in $h$.

Let $B$ be a benzenoid graph. We say that some set of hexagons of $B$ is \textit{independent} (or that the hexagons from this set are independent) if these hexagons are pairwise disjoint. A \textit{Clar set} $C$ is a maximum set of independent $M$-alternating hexagons over all perfect matchings $M$ of $B$. If $C$ is a Clar set and $M$ a perfect matching such that every hexagon from $C$ is $M$-alternating, then we say that the perfect matching $M$ \textit{gives} Clar set $C$. The \textit{Clar number} of $B$, denoted by ${\Cl}(B)$, is the number of hexagons in a Clar set for $B$. It is easy to observe that a Clar set $C$ is a maximum set of independent hexagons such that the  graph obtained from $B$ by removing hexagons from $C$ (and all the edges incident to these hexagons) has a perfect matching.

An \textit{independent set} is a set of vertices in a graph $G$, no two of which are adjacent. A \textit{maximum independent set} is an independent set of largest possible cardinality for a given graph $G$. This cardinality is called the \textit{independence number} of $G$, and denoted by $\alpha(G)$.

A \textit{vertex cover} of a graph is a set of vertices such that each edge of the graph is incident to at least one vertex of the set. Moreover, a graph $G$ is called \textit{subcubic} if the degree of any vertex of $G$ is at most 3.

\section{Auxiliary  results on trees \label{trees}}

In this section some results about the independence number of subcubic trees are proved. These results will be used to establish the main result of the paper. First we prove the following lemma.
\begin{lemma}
\label{l.listi}
Every tree with more than one vertex has a maximum independent set which contains all the leaves.
\end{lemma}
\begin{proof} 
Let $T$ be any tree and $I$ be an independent set of maximum size that contains as many leaves as possible. Suppose that there exists a leaf $u$  of $T$ such that $u\not\in I$.  Denote by $v$ the only neighbour of $u$. We consider the following two cases:
\begin{itemize}
\item if $v \in I$, then the set $(I\setminus \lbrace v \rbrace) \cup \lbrace u \rbrace$ is an independent set of the same size as $I$ and contains $u$, i.e. one more leaf than $I$,
\item if $v \not\in I$, then the set $I \cup \lbrace u \rbrace$ is an independent set of size bigger than the size of $I$. 
\end{itemize}
In both cases, we obtain a contradiction that establishes the lemma. \qed
\end{proof}

\noindent In the next lemma an upper bound for the independence number is shown.

\begin{lemma}
\label{1.istvan}
Let $T$ be a subcubic tree on $n \geq 1$ vertices with independence number $\alpha(T)$. Then
$$\alpha(T) \leq \left[ \frac{2n+1}{3} \right].$$
\end{lemma}

\begin{proof} Let ${\rm VC}$ be a vertex cover of smallest size and let $|{\rm VC}|=\tau(T)$. Since any edge of $T$ is incident with at least one vertex from  ${\rm VC}$, we have $|E(T)| \leq \sum_{v\in {\rm VC}}\deg(v)$. Moreover, $|E(T)| = n-1$, since $T$ is a tree, and $\sum_{v\in {\rm VC}}\deg(v) \leq 3\tau(T)$ since $T$ is subcubic. Hence, $\tau(T) \geq\frac{n-1}{3}$. Substituting this to the Gallai identity $\alpha(T)+\tau(T)=n$ (see \cite{gallai}) we get $\alpha(T)+\frac{n-1}{3}\leq n $, which can be rearranged to the desired form. \qed
\end{proof}

%

%
%
%

\noindent To state the main result of this section, one definition is needed.

\begin{definition} \label{drevesa}
Let $k \geq 2$ be an integer. The tree $T_k$ is composed of the path on vertices $v_1,\ldots,v_{2k+1}$ with $k-2$ additional leaves which are attached to the vertices $v_4, v_6, \ldots, v_{2k-2}$ (see Figure \ref{tree1}). 
\end{definition}

\begin{figure}[!htb]
	\centering
		\includegraphics[scale=0.7, trim=0cm 0cm 0cm 0cm]{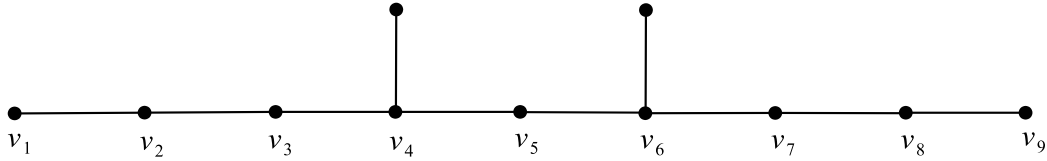}
\caption{Tree $T_4$ from Definition \ref{drevesa}.}
	\label{tree1}
\end{figure}

\noindent In the following theorem we investigate subcubic trees for which the independence number attains the upper bound from Lemma \ref{1.istvan}.

\begin{theorem}
\label{glavna}
Let $T$ be a subcubic tree on $n\ge 3$ vertices with independence number $\left[ \frac{2n+1}{3} \right]$. Then exactly one of the following statements holds:
\begin{itemize}
\item [(i)] $T$ is a tree $T_k$ with $k \geq 2$.
\item [(ii)] $T$ has a vertex that is adjacent to (at least) two leaves.
\end{itemize}

\end{theorem}

\begin{proof}
 Suppose that $T$ satisfies the conditions of the theorem and that $T$ does not have a vertex adjacent to (at least) two leaves. Denote by $\alpha (T)$ the independence number of $T$. Moreover, denote by $\ell$ the number of leaves of $T$. 
Each of these leaves has a unique neighbour, and by assumption all these neighbours are pairwise distinct. We will denote the set of all these neighbours by $S$. Moreover, let $\ell_1$ be the number of vertices in the set $S$ with degree two and let $\ell_2$ be the number of vertices in the set $S$ with degree three. Now, let $T'$ be the forest obtained from $T$ by removing all the leaves and their neighbours. Note that $T'$ has $n'=n-2\ell$ vertices and denote by $r$ the number of connected components of $T'$. It is obvious that $r \leq \ell_2 + 1$. We will denote these components by $C_1, \ldots, C_{r}$ (see Figure \ref{tree2}). Also, let $|V(C_i)| = n_i$ for any $i \in \lbrace 1, \ldots, r \rbrace$.

\begin{figure}[!htb]
	\centering
		\includegraphics[scale=0.9, trim=0cm 0cm 0cm 0cm]{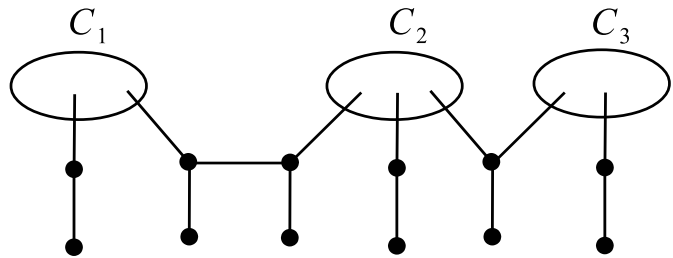}
\caption{Tree $T$ with $\ell = 6$, $\ell_1 = 3$, $\ell_2 = 3$ and connected components $C_1,C_2,C_3$ of $T'$.}
	\label{tree2}
\end{figure}

Since any $C_i$ is a tree, it follows by Lemma \ref{1.istvan} that
$$\alpha(C_i) \leq \frac{2n_i + 1}{3}$$
for any $i \in \lbrace 1, \ldots, r \rbrace$. Hence, we obtain
$$\alpha(T') \leq \sum_{i=1}^{r} \frac{2n_i + 1}{3} = \frac{2n' + r}{3} \leq \frac{2n' + \ell_2 + 1}{3}.$$
Moreover, by Lemma \ref{l.listi} it is easy to observe that
\begin{equation} \label{zveza}
\alpha(T) =\alpha(T') + \ell.
\end{equation}
Therefore, we have
$$\left[ \frac{2n+1}{3} \right] \leq \frac{2n' + \ell_2 + 1}{3} + \ell = \frac{2n - \ell + \ell_2 + 1}{3}.$$
Since $\ell_1 = \ell - \ell_2$, we get
\begin{equation} \label{form}
\left[ \frac{2n+1}{3} \right] \leq \frac{2n - \ell_1 + 1}{3}.
\end{equation}

In the following we first show that $\ell_1 \geq 2$. Let $T''$ be a tree obtained from $T$ by removing all the leaves of $T$. Obviously, $T''$ has more than one vertex (otherwise $T$ has less than three vertices or has a vertex with two leaves, which gives a contradiction in both cases) and therefore, it has at least two leaves. Since it is easy to see that the number of leaves in $T''$ is exactly $\ell_1$, it follows that $\ell_1 \geq 2$. If $\ell_1 \geq 3$, we obtain a contradiction with (\ref{form}) and therefore, $\ell_1 = 2$.

Next, we show that any connected component of $T'$ is a path. Since we already know that the number of leaves in $T''$ is exactly $\ell_1$, it follows that $T''$ has exactly two leaves. Therefore, $T''$ is a path. Since $T'$ is obtained from $T''$ by removing some vertices, it is obvious that every connected component of $T'$ is a path.

Suppose that $T'$ has $m$ isolated vertices ($0 \leq m\leq r \leq \ell_2 + 1$). Therefore, it has $r - m$ connected components isomorphic to a path on more than one vertex. It is obvious that $\alpha(\overline{K_m})=m$, where $\overline{K_m}$ denotes the empty graph on $m$ vertices, and  $\alpha(P_j) \leq 2j/3$ for any $j \geq 2$, where $P_j$ denotes the path on $j$ vertices. Hence, we obtain

$$\alpha(T') \leq m + \frac{2}{3}(n - 2\ell - m) = \frac{2n - 4\ell + m}{3}$$
and by (\ref{zveza}) it follows
$$ \left[ \frac{2n+1}{3} \right] \leq \frac{2n - \ell + m}{3}. $$

\noindent
Obviously, if $\ell \geq m + 2$ we obtain a contradiction with the previous inequality and therefore, $\ell \leq m+1$. Since $m \leq \ell_2 + 1 = \ell - 1$ we also get $\ell \geq m+1$ and we deduce $\ell=m+1$. Hence, $T'$ has $m= \ell - 1 = \ell_2 + 1$ isolated vertices and since the number of connected components of $T'$ is at most $\ell_2 +1$, the graph $T'$ has exactly $\ell_2+1$ connected components and all of them are isolated vertices. Therefore, $T''$ is a path on $2\ell -1$ vertices and finally, we can conclude that the original tree $T$ can be obtained from a path of length $2\ell-1$ by attaching a leaf to every second vertex, starting with the first one, i.e., $T$ is isomorphic to $T_\ell$. Hence, the case $(i)$ of the theorem holds and the proof is complete. \qed
\end{proof}


\begin{lemma}
\label{enolicno}
Let $k \geq 2$. Then for the tree $T_k$ the independence number attains the upper bound from Lemma \ref{1.istvan}. Moreover, the maximum independent set is unique and contains vertices $v_1, v_3, \ldots, v_{2k-1}, v_{2k+1}$ and all the additional leaves.
\end{lemma}

\begin{proof}
Let $I$ be the set of vertices $v_1, v_3, \ldots, v_{2k-1}, v_{2k+1}$ and all the additional leaves. It is easy to check that $I$ contains $2k-1$ vertices and that $T_k$ has exactly $3k-1$ vertices. Therefore,
$$\alpha(T_k)= 2k-1 = \left[ \frac{2(3k-1) + 1}{3} \right]$$
and we have shown that $I$ is the maximum independent set and that for the tree $T_k$ the independence number attains the upper bound from Lemma \ref{1.istvan}.

\begin{figure}[!htb]
	\centering
		\includegraphics[scale=0.7, trim=0cm 0cm 0cm 0cm]{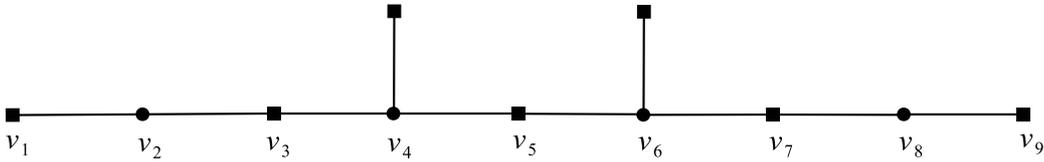}
\caption{Tree $T_4$. Vertices in $I$ are denoted with a square.}
	\label{trees_in}
\end{figure}

To prove the second part, suppose that $I'$ is another maximum independent set in $T_k$. Since $I' \neq I$ there is some $u \in I'$ such that $u \notin I$. We consider two cases.
\begin{itemize}
\item [(a)] \textit{$u$ is a vertex of degree 3.} Let $T_k'$ be the graph obtained from $T_k$ by removing $u$ and all its neighbours. Obviously, $T_k'$ has two connected components and we will denote them by $C_0$ and $C_1$. Let $I_j = I \cap V(C_j)$ for $j \in \lbrace 0, 1\rbrace$. It is easy to check that $I_j$ is a maximum independent set for $C_j$ where $j \in \lbrace 0, 1\rbrace$. Therefore, $|I'| = 1 + \alpha(C_0) + \alpha(C_1)$. On the other hand, we know that $|I| = 3 + \alpha(C_0) + \alpha(C_1)$. Since $I'$ and $I$ are both maximum independent sets, it holds $|I| = |I'|$ and we get a contradiction.
\item [(b)] \textit{$u=v_2$ or $u = v_{2k}$.} Let $T_k'$ be the graph obtained from $T_k$ by removing $u$ and all its neighbours. Let $I'' = I \cap V(T_k')$.  It is easy to check that $I''$ is a maximum independent set for $T_k'$. Therefore, $|I'| = 1 + \alpha(T_k')$. On the other hand, we know that $|I| = 2 + \alpha(T_k')$. Since $I'$ and $I$ are both maximum independent sets, it holds $|I| = |I'|$ and we get a contradiction.

\end{itemize}
\noindent
Since we get a contradiction in every case it follows that $I$ is the unique maximum independent set for $T_k$. \qed
\end{proof}

\section{Catacondensed benzenoid graphs with large Clar number}
\label{bound}

In this section we prove an upper bound for the Clar number of catacondensed benzenoid graphs and characterize those graphs that attain this bound. 

The \textit{dualist} graph of a given benzenoid graph $B$ consists of vertices corresponding to hexagons of $B$; two vertices are adjacent if and only if the corresponding hexagons have a common edge.
Obviously, the dualist graph of $B$ is a tree if and only if $B$ is catacondensed. If $B$ has $n$ hexagons, then this tree has $n$ vertices and none of its vertices have degree greater than 3 (so it is a subcubic tree). For a catacondensed benzenoid graph $B$ we will denote its dualist tree by $T(B)$. For an example see Figure \ref{dual}.

\begin{figure}[!htb]
	\centering
		\includegraphics[scale=0.7, trim=0cm 0cm 0cm 0cm]{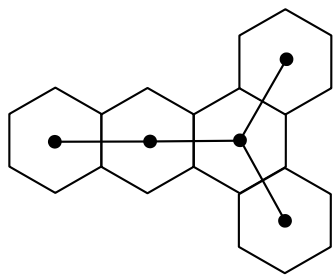}
\caption{Benzenoid graph $B$ with the dualist tree $T(B)$.}
	\label{dual}
\end{figure}

\begin{lemma} \label{zg_clar}
Let $B$ be a catacondensed benzenoid graph with $n$ hexagons. Then
$${\Cl}(B) \leq \left[ \frac{2n+1}{3} \right].$$
\end{lemma}

\begin{proof}
Let $C$ be a Clar set for $B$ and let $T(B)$ be a dualist tree of $B$. Obviously, the vertices corresponding to the hexagons of $C$ form an independent set in $T(B)$. Therefore, by Lemma \ref{1.istvan} we have
$${\Cl}(B) \leq \alpha(T(B)) \leq \left[ \frac{2n+1}{3} \right].$$
\qed
\end{proof}

A hexagon $h$ of a benzenoid graph $B$ adjacent to exactly two other hexagons possesses two vertices of degree 2. If these two vertices are adjacent, then $h$ is \textit{angularly connected}, for short we say that $h$ is \textit{angular}. If these two vertices are
not adjacent, then $h$ is \textit{linearly connected}, and we say that $h$ is \textit{linear}.

\noindent
To characterize graphs that attain the upper bound we need the following lemma. 

\begin{lemma} \label{pom1}
Let $B$ be a catacondensed benzenoid graph with $n$ hexagons such that $T(B) \simeq T_k$ for some $k \geq 2$. Then ${\Cl}(B) = \left[ \frac{2n+1}{3} \right]$ if and only if the two hexagons corresponding to vertices $v_2$ and $v_{2k}$ are both angular.
\end{lemma}

\begin{proof}
First suppose that the two hexagons corresponding to vertices $v_2$ and $v_{2k}$ are both angular and we denote these hexagons by $h$ and $h'$. It is obvious that any vertex of $T_k$ different from $v_2$ and $v_{2k}$ is either in the unique maximum independent set or has degree 3. We need to find a perfect matching $M$ for $B$ with exactly $\left[ \frac{2n+1}{3} \right]$ independent $M$-alternating hexagons. Let $M'$ be a matching containing exactly 3 edges from all the hexagons that correspond to the vertices in the unique maximum independent set of $T_k$. Moreover, let $e$ and $e'$ be the edges of $h$ and $h'$, respectively, with both end-vertices of degree 2. Finally, we define $M = M' \cup \lbrace e,e' \rbrace$ (see Figure \ref{ben_graf}). It is easy to check that $M$ is a perfect matching for $B$ and by Lemma \ref{enolicno} it has exactly $ \left[ \frac{2n+1}{3} \right]$ independent $M$-alternating hexagons. Therefore, ${\Cl}(B) = \left[ \frac{2n+1}{3} \right]$.

\begin{figure}[!htb]
	\centering
		\includegraphics[scale=0.7, trim=0cm 0cm 0cm 0cm]{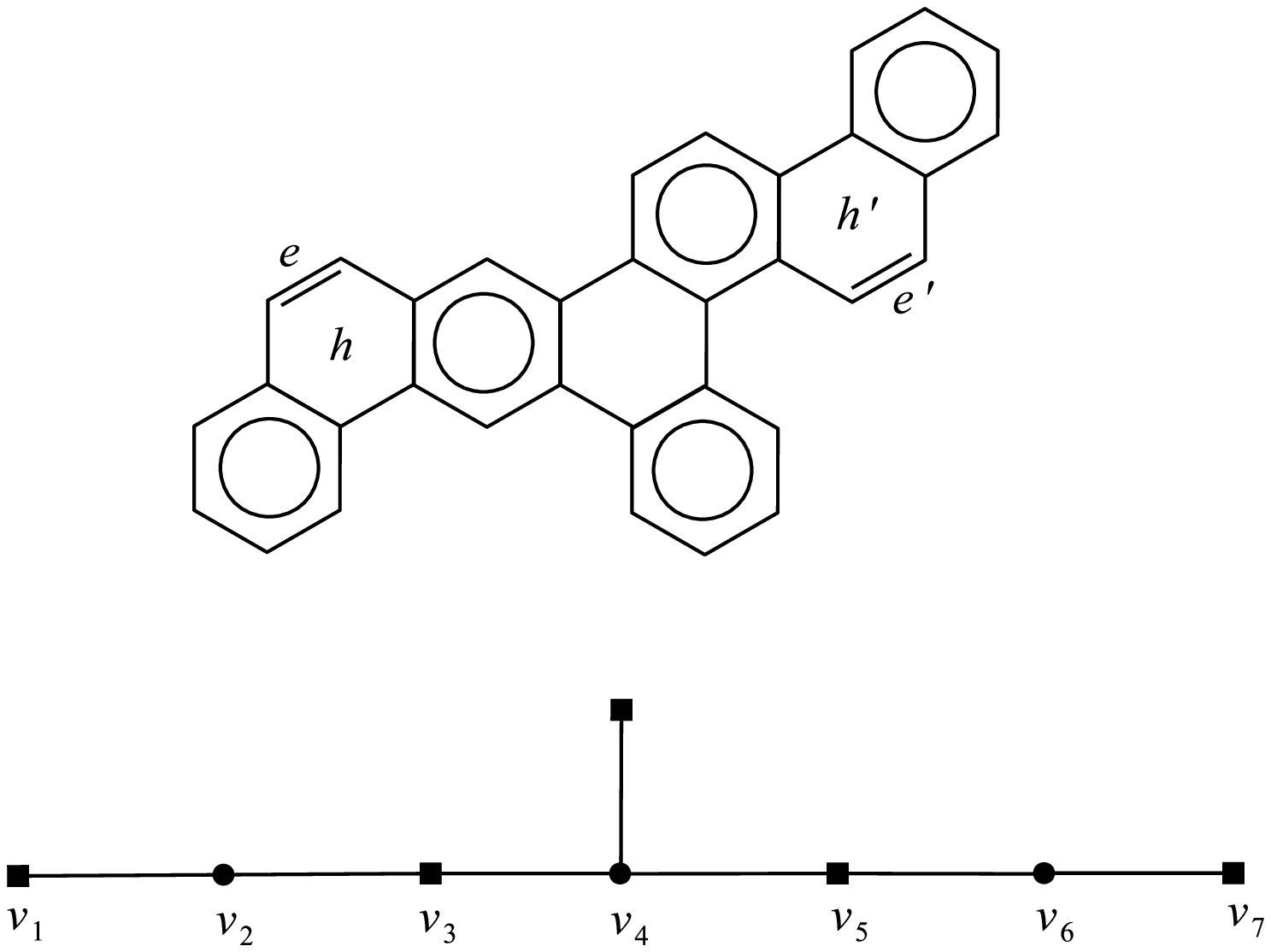}
\caption{Benzenoid graph with a perfect matching $M$ and the corresponding dualist tree $T_3$.}
	\label{ben_graf}
\end{figure}

To show the other direction, suppose that ${\Cl}(B) = \left[ \frac{2n+1}{3} \right]$. Hence, there is a Clar set $C$ for $B$ with exactly $\left[ \frac{2n+1}{3} \right]$ hexagons. The vertices of $T(B) \simeq T_k$ corresponding to the hexagons from $C$ form an independent set $I$ of $T_k$. Since $|I| = \left[ \frac{2n+1}{3} \right]$, set $I$ is uniquely defined by Lemma \ref{enolicno}. Let $M$ be a perfect matching for $B$ that gives Clar set $C$. Since $v_1, v_3 \in I$, the hexagons corresponding to the vertices $v_1$ and $v_3$ are $M$-alternating hexagons. If the hexagon corresponding to vertex $v_2$ is not angular, there are two vertices in this hexagon that are not incident with an edge from $M$, which is a contradiction. Therefore, the hexagon corresponding to $v_2$ is angular. In a similar way we can show that the hexagon corresponding to $v_{2k}$ is also angular and the proof is complete. \qed 
\end{proof}

Next, we define family $\mathcal{B}$ of catacondensed benzenoid graphs as follows:

\begin{itemize}
\item [$(i)$] The benzenoid graph with one hexagon belongs to $\mathcal{B}$ and the benzenoid graph with two hexagons belongs to $\mathcal{B}$. If $B_1$ is a catacondensed benzenoid graph with three hexagons such that the hexagon adjacent to two other hexagons is angular, then $B_1$ belongs to   $\mathcal{B}$ (see Figure \ref{grafi}).

\begin{figure}[!htb]
	\centering
		\includegraphics[scale=0.7, trim=0cm 0cm 0cm 0cm]{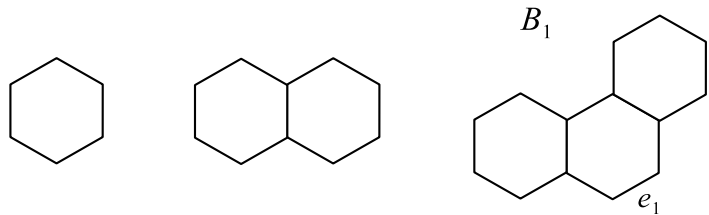}
\caption{Benzenoid graphs from case $(i)$.}
	\label{grafi}
\end{figure}

\item [$(ii)$] Let $B$ be a catacondensed benzenoid graph such that $T(B) \simeq T_k$ for some $k \geq 2$ and such that the two hexagons corresponding to vertices $v_2$ and $v_{2k}$ are both angular. Then $B$ belongs to $\mathcal{B}$.
\item [$(iii)$] Let $B'$ be a catacondensed benzenoid graph from $\mathcal{B}$ and let $e'$ be any edge of $B'$ with both end-vertices of degree 2. Moreover, let $B_1$ be the benzenoid graph from Figure \ref{grafi} and let $e_1$ be the edge of the angular hexagon with both end-vertices of degree 2. We define $B$ to be a benzenoid graph obtained from $B'$ and $B_1$ by identifying edges $e'$ and $e_1$ (see Figure \ref{lepljenje}). Then $B$ belongs to $\mathcal{B}$.

\begin{figure}[!htb]
	\centering
		\includegraphics[scale=0.7, trim=0cm 0cm 0cm 0cm]{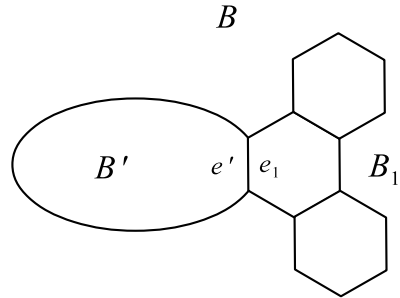}
\caption{Benzenoid graph $B$ obtained from $B' \in \mathcal{B}$ and $B_1$.}
	\label{lepljenje}
\end{figure}

\end{itemize}

\noindent
We notice that family $\mathcal{B}$ is defined inductively. Cases $(i)$ and $(ii)$ represent the basis and Case $(iii)$ represents the inductive step. It is easy to observe that $\mathcal{B}$ contains only catacondensed benzenoid graphs. Finally, we are able to prove the main result of this paper.

\begin{theorem}  \label{gl_iz}
Let $B$ be a catacondensed benzenoid graph with $n$ hexagons. Then
$${\Cl}(B) \leq \left[ \frac{2n+1}{3} \right]$$
and equality holds if and only if $B$ belongs to $\mathcal{B}$.
\end{theorem}

\begin{proof}
The inequality follows by Lemma \ref{zg_clar}. 

First, suppose that a benzenoid graph $B$ with $n$ hexagons belongs to $\mathcal{B}$. We will show that the Clar number of $B$ attains the upper bound. Consider the following cases:
\begin{itemize}
\item [$(i)$] If $B$ has one hexagon, then ${\Cl}(B)=1 = \left[ \frac{2 \cdot 1 + 1}{3}\right]$. If $B$ has two hexagons, then ${\Cl}(B)=1 = \left[ \frac{2 \cdot 2 + 1}{3}\right]$. Also, for $B_1$ it holds ${\Cl}(B_1)=2 = \left[ \frac{2 \cdot 3 + 1}{3}\right]$.
\item [$(ii)$]  If $B$ is a catacondensed benzenoid graph such that $T(B) \simeq T_k$ for some $k \geq 2$ and such that the two hexagons corresponding to vertices $v_2$ and $v_{2k}$ are both angular, then the Clar number of $B$ attains the upper bound by Lemma \ref{pom1}.
\item [$(iii)$] Let $B$ be obtained from $B' \in \mathcal{B}$ and $B_1$ by identifying edges $e'$ and $e_1$ (see Figure \ref{lepljenje}). We use induction to prove that the Clar number of $B$ attains the upper bound. The base step is proved in Case $(i)$ and Case $(ii)$. For the induction step, assume that $B'$ has $n'$ hexagons and that the Clar number of $B'$ attains the upper bound, i.e.\,${\Cl}(B') = \left[ \frac{2n'+1}{3} \right].$ Hence, let $M'$ be a perfect matching of $B'$ with ${\Cl}(B')$ independent $M'$-alternating hexagons. Moreover, let $M''$ be a perfect matching of $B_1$ with two independent $M''$-alternating hexagons. Obviously, edge $e_1$ of $B_1$ belongs to $M''$. Finally, we define $M = M' \cup (M'' \setminus \lbrace e_1 \rbrace)$. Obviously, $M$ is a perfect matching of $B$ with exactly ${\Cl}(B')+2$ independent $M$-alternating hexagons. Hence, we obtain $${\Cl}(B) \geq {\Cl}(B') + 2 = \left[ \frac{2n'+1}{3}  \right] + 2 = \left[ \frac{2(n'+3) + 1}{3}\right].$$
Since $B$ has $n'+3$ hexagons, by Lemma \ref{zg_clar} we get
$${\Cl}(B) = \left[ \frac{2(n'+3) + 1}{3}\right].$$
\end{itemize}
\noindent
We have shown that if $B$ belongs to $\mathcal{B}$, then the Clar number of $B$ attains the upper bound.

For the other direction, suppose that $B$ has $n$ hexagons and that ${\Cl}(B) = \left[ \frac{2n+1}{3} \right]$. We have to prove that $B$ belongs to $\mathcal{B}$. Let $C$ be a Clar set of $B$ and let $I$ be the set of vertices in $T(B)$ that correspond to the hexagons from $C$. Obviously, by Lemma \ref{1.istvan}, $\alpha(T(B)) = |I| = \left[ \frac{2n+1}{3} \right]$. We consider two cases:
\begin{itemize}
\item [(a)] \textit{$T(B)$ does not have a vertex that is adjacent to (at least) two leaves.} By Theorem \ref{glavna}, $T(B)$ has at most two vertices or $T(B) \simeq T_k$ for some $k \geq 2$. In the last case, by Lemma \ref{pom1}, the two hexagons corresponding to vertices $v_2$ and $v_{2k}$ are both angular. In both cases, $B$ belongs to $\mathcal{B}$.
\item [(b)]\textit{ $T(B)$ has a vertex that is adjacent to (at least) two leaves.} If $T(B)$ has three vertices, then $B$ has three hexagons and the hexagon adjacent to two other hexagons is angular (otherwise the Clar number does not attain the upper bound). Therefore, $B=B_1$ and $B$ belongs to $\mathcal{B}$.

Now, suppose that $T(B)$ has more than three vertices. Then we have the situation from Figure \ref{odstrani}. Let $M$ be a perfect matching of $B$ that gives the Clar set $C$. We will show that $h_2,h_3 \in C$. We will also denote by $B_1$ the subgraph of $B$ composed of hexagons $h_1$, $h_2$, and $h_3$. Consider the following cases:

\begin{figure}[!htb]
	\centering
		\includegraphics[scale=0.7, trim=0cm 0cm 0cm 0cm]{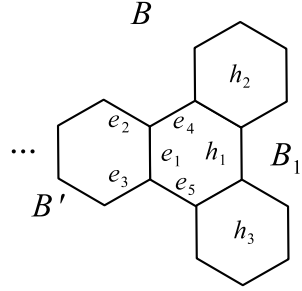}
\caption{Benzenoid graph $B$ with edges $e_1,e_2,e_3,e_4$, and $e_5$.}
	\label{odstrani}
\end{figure}

\begin{itemize}
\item [1.] If $e_2,e_5 \in M$ or $e_3,e_4 \in M$, then we can easily see that $M$ can not be a perfect matching, so this case can not happen.
\item [2.] If $e_4,e_5 \in M$, then hexagons $h_2,h_3$ are not $M$-alternating and perfect matching $M$ can be changed to $\overline{M}$ (in subgraph $B_1$) so that $h_2,h_3$ are $\overline{M}$-alternating. Therefore, $\overline{C} = (C \setminus \lbrace h_1\rbrace) \cup \lbrace h_2,h_3\rbrace $ is a set of independent $\overline{M}$-alternating hexagons and $|\overline{C}| > |C|$, which is a contradiction.
\item [3.] If $e_1 \in M$ or $e_2,e_3 \in M$, then $h_2,h_3$ are $M$-alternating and $h_2,h_3 \in C$ (otherwise $C$ is not a Clar set).
\end{itemize}
Since cases 1.\,and 2.\,can not happen, it follows that $h_2,h_3$ are $M$-alternating and $h_2,h_3 \in C$. We define $B'$ to be the graph obtained from $B$ by removing hexagons $h_2$ and $h_3$ and all the edges incident to $h_2$ or $h_3$. Moreover, let $M' = M \cap E(B')$. Obviously, $C'=C \setminus \lbrace h_2, h_3\rbrace $ is a set of independent $M'$-alternating hexagons and therefore,
$${\Cl}(B') \geq \left[ \frac{2n+1}{3} \right] - 2 = \left[ \frac{2(n-3)+1}{3} \right].$$
We know that $B'$ has exactly $n-3$ hexagons and hence, by Lemma \ref{zg_clar}, ${\Cl}(B') = \left[ \frac{2(n-3)+1}{3} \right]$.
If $T(B')$ does not have a vertex that is adjacent to two leaves, then by Case (a) $B'$ belongs to $\mathcal{B}$ and therefore, by definition of $\mathcal{B}$, also $B$ belongs to $\mathcal{B}$.
If $T(B')$ has a vertex that is adjacent to two leaves, then we can repeat the above procedure (at every step we remove 3 hexagons) until we get $B''$ such that $T(B'')$ does not have a vertex that is adjacent to two leaves. Again, it follows that $B$ belongs to $\mathcal{B}$.

\end{itemize}
Since in every case we get that $B$ belongs to $\mathcal{B}$, the proof is finished. \qed
\end{proof}

We conclude this section with two additional results.
\begin{lemma} \label{pom2}
Let $B$ be a catacondensed benzenoid graph and $h_0$ a hexagon that is adjacent to at most one other hexagon. Moreover, let $B'$ be a benzenoid graph obtained from $B$ by adding $k$, $k \geq 1$, linearly connected hexagons to $h_0$ such that $h_0$ is also linearly connected (or adjacent only to $h_1$), see Figure \ref{vstav}. Then ${\Cl}(B')={\Cl}(B)$.
\end{lemma}

\begin{proof}
Denote the added hexagons by $h_1, \ldots, h_k$ and let $C$ be a Clar set for $B$. See Figure \ref{vstav}.

\begin{figure}[!htb]
	\centering
		\includegraphics[scale=0.7, trim=0cm 0cm 0cm 0cm]{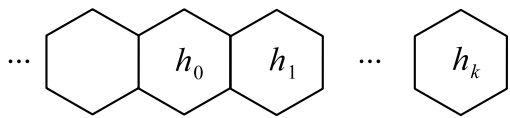}
\caption{Adding hexagons $h_1, \ldots, h_k$.}
	\label{vstav}
\end{figure}

Obviously, there is a perfect matching $M'$ of $B'$ such that $C$ is an independent set of $M'$-alternating hexagons. Therefore, it follows ${\Cl}(B') \geq {\Cl}(B)$.

To show the equality, let $C'$ be a Clar set for $B'$. Since hexagons $h_0,h_1,\ldots,h_k$ are linearly connected, at most one of these hexagons belongs to $C'$. Consider two cases.
\begin{itemize}
\item If one of hexagons $h_1,\ldots,h_k$ belongs to $C'$, then it is possible to find a Clar set $C''$ for $B'$ such that $h_0 \in C''$. Obviously, $C''$ is also a Clar set for $B$ and ${\Cl}(B')={\Cl}(B)$.
\item If none of hexagons $h_1,\ldots,h_k$ belongs to $C'$, then $C'$ is also a Clar set for $B$ and ${\Cl}(B')={\Cl}(B)$.
\end{itemize}
In both cases we get ${\Cl}(B')={\Cl}(B)$ and the proof is complete. \qed
\end{proof}

\noindent
Next result follows by Lemma \ref{pom2} and the main theorem.

\begin{theorem} Let $n$ be a positive integer. Then for every $c \in \left \lbrace 1,2,\ldots,\left[ \frac{2n+1}{3} \right] \right \rbrace$ there exists
a catacondensed benzenoid graph $B$ with $n$ hexagons and with Clar number ${\Cl}(B)=c$.
\end{theorem}

\begin{proof}
Let $n$ be a positive integer and  $c \in \left \lbrace 1,2,\ldots,\left[ \frac{2n+1}{3} \right] \right \rbrace$. If $c = \left[ \frac{2n+1}{3} \right]$, then such a graph exists by Theorem \ref{gl_iz}. Now suppose that $c \leq \left[ \frac{2n+1}{3} \right] - 1$. Then we obtain $c \leq \frac{2n-2}{3}$.

Let $B'$ be a catacondensed benzenoid graph with $n'$ hexagons such that $c= \left[ \frac{2n'+1}{3} \right]$. Hence, $c > \frac{2n'-2}{3}$ and $n' < \frac{3c+2}{2}$. Finally, we get
$$n' <  \frac{3c+2}{2} \leq \frac{3 \cdot \frac{2n-2}{3} + 2}{2} = n.$$

\noindent
Therefore, we add $n-n'$ linearly connected hexagons to one hexagon (which corresponds to a vertex of degree at most one in $T(B')$) of $B'$ to obtain $B$. Obviously, $B$ has $n$ hexagons and by Lemma \ref{pom2}, ${\Cl}(B)=c$. \qed
\end{proof}

\section*{Acknowledgement}

\noindent The author Istv\'{a}n Est\'{e}lyi was supported by the Czech Ministry of Education, Youth and
Sports (project LO1506). The author Riste \v Skrekovski acknowledge the financial support from the Slovenian Research Agency (research core funding No. P1-0383). The author Niko Tratnik was finacially supported by the Slovenian Research Agency.


\baselineskip=16pt




\begin{thebibliography}{99}


%
%
%
%
%
%
%
%
%

\bibitem{abeledo} H. Abeledo, G. W. Atkinson, Unimodularity of the Clar number problem, \textit{Linear Algebra Appl.} \textbf{420} (2007) 441--448.

\bibitem{ahmadi} M. B. Ahmadi, E. Farhadi, V. Amiri Khorasani, On computing the Clar number of a fullerene using
optimization techniques, \textit{MATCH Commun. Math. Comput. Chem.} \textbf{75} (2016) 695--701.

\bibitem{ahmadi2} M. B. Ahmadi, V. A. Khorasani, E. Farhadi, Saturation number of fullerene and benzenoid graphs, \textit{MATCH Commun. Math. Comput. Chem.} \textbf{77} (2017) 737--747.

\bibitem{clar} E. Clar, \textit{The Aromatic Sextet}, John Wiley \& Sons, London, 1972.
	
\bibitem{hos2} R. Cruz, C. A. Marin, J. Rada, Computing the Hosoya index of catacondensed hexagonal systems, \textit{MATCH Commun. Math. Comput. Chem.} \textbf{77} (2017) 749--764.


\bibitem{DGKZ-2002}
	A.~A.~Dobrynin, I.~Gutman, S.~Klav\v zar, P.~\v Zigert,
	Wiener index of hexagonal systems,
	\textit{Acta Appl. Math.} \textbf{72} (2002) 247--294.
	
\bibitem{dos} T. Dosli\' c, N. Tratnik, D. Ye, P. \v Zigert Pleter\v sek, On 2-cores of resonance graphs of fullerenes, \textit{MATCH Commun. Math. Comput. Chem.} \textbf{77} (2017) 729--736.

\bibitem{gallai} T. Gallai, {\"U}ber extreme Punkt- und Kantenmengen, 
\textit{Ann. Univ. Sci. Budapest. E{\"o}tv{\"o}s Sect. Math.} \textbf{2} (1959) 133--138.

\bibitem{gra1} J. E. Graver, E. J. Hartung, Internal Kekul\' e structures for graphene and general patches, \textit{MATCH Commun. Math. Comput. Chem.} \textbf{76} (2016) 693--705.

\bibitem{gra} J. E. Graver, E. J. Hartung, A. Y. Souid, Clar and Fries numbers for benzenoids, \textit{J. Math. Chem.} \textbf{51} (2013) 1981--1989.

\bibitem{gucy-89}
I.  Gutman, S.~J. Cyvin. 
{\it Introduction to the Theory of Benzenoid Hydrocarbons},
Springer-Verlag, Berlin, 1989.

\bibitem{ha} P. Hansen, M. Zheng, The Clar number of a benzenoid hydrocarbon and linear programming, \textit{J. Math. Chem.} \textbf{15} (1994) 93--107.

\bibitem{ha2} P. Hansen, M. Zheng, Upper bounds for the Clar number of a benzenoid hydrocarbon, \textit{J. Chem. Soc., Faraday Trans.} \textbf{88} (1992) 1621--1625.

\bibitem{har1} E. Hartung, Clar chains and a counterexample, \textit{J. Math. Chem.} \textbf{52} (2014) 990--1006.


\bibitem{har} E. Hartung, Fullerenes with complete Clar structure, \textit{Discrete Appl. Math.} \textbf{161} (2013) 2952--2957.
	
\bibitem{kl-zi-gu} S. Klav\v zar, P. \v Zigert, I. Gutman, Clar number of catacondensed benzenoid hydrocarbons, \textit{J. Mol. Struc. THEOCHEM} \textbf{586} (2002) 235--240.

\bibitem{sa-gu} K. Salem, I. Gutman, Clar number of hexagonal chains, \textit{Chem. Phys. Letters} \textbf{394} (2004) 283--286.

\bibitem{sa-kl} K. Salem, S. Klav\v zar, A. Vesel, P. \v Zigert, The Clar formulas of a benzenoid system and the resonance graph, \textit{Discrete Appl. Math.} \textbf{157} (2009) 2565--2569.

\bibitem{shi} L. Shi, H. Zhang, Forcing and anti-forcing numbers of $(3,6)$-Fullerenes, \textit{MATCH Commun. Math. Comput. Chem.} \textbf{76} (2016) 597--614.

\bibitem{ewi} Z. F. Wei, H. Zhang, Number of matchings of low order in $(4,6)$-fullerene graphs, \textit{MATCH Commun. Math. Comput. Chem.} \textbf{77} (2017) 707--724.

\bibitem{zhang_pol} H. A. Witek, J. Langner, G. Mos, C. P. Chou, Zhang-Zhang polynomials of regular 5-tier benzenoid strip, \textit{MATCH Commun. Math. Comput. Chem.} \textbf{78} (2017) 487--504.

\bibitem{ye-zh} D. Ye, H. Zhang, Extremal fullerene graphs with the maximum Clar number, \textit{Discrete Appl. Math.} \textbf{157} (2009) 3152--3173.

\bibitem{hos} J. Zhang, X. Chen, W. Sun, A linear-time algorithm for the Hosoya index of an arbitrary tree, \textit{MATCH Commun. Math. Comput. Chem.} \textbf{75} (2016) 703--714.

\bibitem{zh-ye} H. Zhang, D. Ye, An Upper Bound for the Clar Number of Fullerene Graphs, \textit{J. Math. Chem.} \textbf{41} (2007) 123--133.
  
  
%
%
%
%
%
%
%
%
%
%
%
%
%
%
%
%
%
%
%
%
%
%
%
%
%
%
%
%
%
    
\end{thebibliography}
\end{document}